\documentclass[12pt]{article}

\usepackage{mathrsfs}

\usepackage{amsmath,amsthm,amssymb}
\usepackage{graphicx}
\usepackage{vector}

\font\tencmmib=cmmib10 \skewchar\tencmmib '60
\newfam\cmmibfam
\textfont\cmmibfam=\tencmmib

\def\lessim{\ \lower4pt\hbox{$
\buildrel{\displaystyle <}\over\sim$}\ }
\def\gessim{\ \lower4pt\hbox{$\buildrel{\displaystyle >}
\over\sim$}\ }

\newcommand{\e}{\mathbb{E}}
\newcommand{\p}{\mathbb{P}}

\newtheorem{lemma}{\bf Lemma}

\newtheorem{theorem}{\bf Theorem}
\newtheorem{corollary}{\bf Corollary}

\newtheorem{example}{\bf Example}

\newtheorem{proposition}{\bf Proposition}

\newenvironment{Proof of lemma}{\noindent{\bf Proof of Lemma}}{\hfill$\Box$\newline}
\newenvironment{Proof of theorem}{\noindent{\bf Proof of Theorem}}{\hfill{\footnotesize${\square}$}\newline}
\newenvironment{Proof of theorems}{\noindent{\bf Proof of Theorems}}{\hfill$\Box$\newline}
\newenvironment{Proof of proposition}{\noindent{\bf Proof of Proposition}}{\hfill$\Box$\newline}
\newenvironment{Proof of propositions}{\noindent{\bf Proof of Propositions}}{\hfill$\Box$\newline}
\newenvironment{Proof of exercise}{\noindent{\it Proof of Exercise:}}{\hfill$\Box$}
\font\tencmmib=cmmib10 \skewchar\tencmmib '60
\newfam\cmmibfam
\textfont\cmmibfam=\tencmmib

\def\lessim{\ \lower4pt\hbox{$
\buildrel{\displaystyle <}\over\sim$}\ }
\def\gessim{\ \lower4pt\hbox{$\buildrel{\displaystyle >}
\over\sim$}\ }

\def\go0{\to 0}

\def\leftitem#1{\item{\hbox to\parindent{\enspace#1\hfill}}}

\def\sg{\sigma}

\def\sg2{\sigma^2}

\def\__{_{\infty}}
\begin{document}

\title{Partial results on the convexity of the Parisi functional with PDE approach }
\author{Wei-Kuo Chen\footnote{Department of Mathematics, University of Chicago, email: wkchen@math.uchicago.edu}}

\maketitle

\begin{abstract}
We investigate the convexity problem for the Parisi functional defined on the space of the so-called functional ordered parameters in the Sherrington-Kirkpatrick model. In the recent work of Panchenko \cite{P08}, he proved that this functional is convex along one-sided directions with a probabilistic method. In this paper, we will study this problem with a PDE approach that simplifies his original proof and presents more general results. 
\end{abstract}

\section{Introduction and main results}
The Sherrington-Kirkpatrick (SK) model was introduced in \cite{SK75} with the aim of explaining the strange magnetic behaviors of certain alloys. In the past decades, it has been intensively studied in physics, see \cite{MPV}. One of the most beautiful discoveries was the famous Parisi formula, which states the thermodynamic limit of the free energy in the SK model with inverse temperature $\beta>0$ and external field $h\in\mathbb{R}$ can be computed through a variational problem over the space of functional ordered parameters $\mathcal{M}$,
\begin{align}
\label{pf}
\inf_{a\in\mathcal{M}}\left(\log 2+F_{a}(\beta h,1)-\frac{\beta^2}{2}\int_0^1ta(1-t)dt\right),
\end{align}
where $\mathcal{M}$ is defined as the collection of all nonincreasing and left-continuous functions from $[0,1]$ to $[0,1]$ and $F_a$ is the solution to the nonlinear parabolic PDE associated with $a\in\mathcal{M}$, 
\begin{align}
\begin{split}
\label{pde1}
\partial_tF_{a}(x,t)&=\frac{1}{2}\left(\partial_{xx}F_{a}(x,t)+a(t)(\partial_xF_{a}(x,t))^2\right),\,\,(x,t)\in\mathbb{R}\times[0,1],
\end{split}\\
\begin{split}
\label{ic}
F_a(x,0)&=\log \cosh(x),\,\,x\in\mathbb{R}.
\end{split}
\end{align}
Note that if $a$ is a step function, $F_a$ can be precisely solved by performing the Hopf-Cole transformation, while for general $a$, the existence of $F_a$ is assured via an approximation argument using step functions and the uniform $L^1$-Lipschitz property of $a\mapsto F_a(x,t)$ over arbitrary $(x,t)$, see \cite{Guerra03}. The first mathematically rigorous proof for the Parisi formula was presented in the seminal work of Talagrand \cite{Tal06}. Later, its validity in the mixed $p$-spin model was established by Panchenko \cite{P12}. 

%See also its extended versions in the spherical setting in Talagrand \cite{Tal06} and Chen \cite{C12} as well as in the multidimensional-spin models in %Bovier and Klimovsky \cite{KB09}. 

\smallskip

It has been conjectured that the Parisi formula \eqref{pf} has a unique minimizer. In physicists' picture, such minimizer is encoded with all information that are needed to completely describe the model. As the third term in the bracket of \eqref{pf} is linear in the functional ordered parameter, the uniqueness will follow if one could show that $a\mapsto F_a(\beta h,1)$ defines a convex functional on $\mathcal{M}$. More generally, let $\phi$ be a twice differentiable function on $\mathbb{R}$ with
\begin{align}\label{eq-9}
\|\phi'\|_\infty\,\,\mbox{and}\,\,\|\phi''\|_\infty<\infty.
\end{align} 
These assumptions guarantee the existence of the solution $F_{\phi,a}$ corresponding to the PDE \eqref{pde1} and step-like $a\in\mathcal{M}$ with a new initial condition $\phi$ in \eqref{ic}. They also allow us to adapt a similar argument as Proposition 3.1 in \cite{Tal07} to obtain the uniform $L^1$-Lipschitz property of $a\mapsto F_{\phi,a}(x,t)$ on all step-like $a\in\mathcal{M}$ over $(x,t)$. So the existence of $F_{\phi,a}$ for arbitrary $a\in\mathcal{M}$ is ensured. Now for any $x,$ define the Parisi functional $P_{\phi,x}$ on $\mathcal{M}$ as $P_{\phi,x}(a):=F_a(x,1)$ for $a\in\mathcal{M}$. Note that $\log \cosh x$ is an even convex function and satisfies \eqref{eq-9}. Numerical simulation suggests
\smallskip
\smallskip

\noindent{\bf Conjecture.} {\it The Parisi functional $P_{\phi,x}$ is convex on $\mathcal{M}$ if $\phi$ is even convex.}

\smallskip
\smallskip
\noindent 
The first related result about this problem was presented in Panchenko \cite{P08}, where using a probabilistic argument, he showed the convexity along one-sided directions, that is, 
\begin{align}
\label{pan}
P_{\phi,x}(\alpha a_1+(1-\alpha)a_2)\leq \alpha P_{\phi,x}(a_1)+(1-\alpha)P_{\phi,x}(a_2),\,\,\forall \alpha\in[0,1]
\end{align}
for all $a_1,a_2\in\mathcal{M}$ with $a_1\leq a_2$. In this paper, we will study the above conjecture via a maximum principle for the nonlinear parabolic PDE \eqref{pde1}. Although at this point we still have not been able to figure out how to use the present method to give a complete answer, it provides a new way of looking at the convexity problem that simplifies Panchenko's original argument and leads to more general results. Incidentally, it is of independent interest that the numerical evidence seems to support that the conclusion of Conjecture also holds true when $\phi$ is nondecreasing. As one shall see, our method can be as well applied to this case and deduce similar convexity along one-sided directions. 

\smallskip

We now state the main results. Let $\mathcal{C}$ be the collections of all twice differentiable $\phi$ on $\mathbb{R}$ satisfying \eqref{eq-9}. Set function spaces
\begin{align*}
\mathcal{F}_1&=\{(\phi_1,\phi_2)|\mbox{$\phi_1,\phi_2\in \mathcal{C}$ even convex with $\phi_1\leq \phi_2$ and $\phi_1'\leq \phi_2'$}\},\\
\mathcal{F}_2&=\{(\phi_1,\phi_2)|\mbox{$\phi_1,\phi_2\in\mathcal{C}$ nondecreasing convex with $\phi_1\leq \phi_2$}\\
&\qquad\qquad\qquad\mbox{and $\phi_1'(x)\leq \phi_2'(x)$ $\forall x\geq 0$}\},\\
\mathcal{F}&=\mathcal{F}_1\cup \mathcal{F}_2.
\end{align*}

\begin{theorem}
\label{thm2}
Suppose that $(\phi_1,\phi_2)\in\mathcal{F}.$ If $a_1,a_2\in\mathcal{M}$ with $a_1\leq a_2,$ then 
\begin{align}
\label{thm2:eq1}
P_{\phi,x}(a)\leq \alpha P_{\phi_1,x}(a_1)+(1-\alpha)P_{\phi_2,x}(a_2)
\end{align}
for all $\alpha\in[0,1]$ and $x\in\mathbb{R},$
where $\phi:=\alpha\phi_1+(1-\alpha)\phi_2$ and $a:=\alpha a_1+(1-\alpha)a_2.$
\end{theorem}

Letting $(\phi_1,\phi_2)\in\mathcal{F}$ with $\phi_1=\phi_2$, we obtain the convexity along the one-sided directions. 

\begin{corollary}
If $\phi\in\mathcal{C}$ is either nondecreasing convex or even convex, then \eqref{pan} holds.
\end{corollary}

\begin{example}
\rm Let $\phi$ be either nondecreasing convex or even convex. For any given $m\in [0,1],$ if $a(t)=m$ for all $0\leq t\leq 1,$ then $P_{\phi,x}(a)=m^{-1}\log \e\exp m\phi(x+z)$, where $z$ is a standard Gaussian random variable. Thus, Corollary 1 implies that $m\mapsto m^{-1}\log \e\exp m\phi(x+z)$ defines a convex function on $[0,1].$ Let us emphasize that this is different from the usual log-convexity of the $L^m$-norm in $1/m.$ 
\end{example}

\section{Proofs}

Throughout this paper, we denote by $z$ the standard Gaussian random variable and set $z_{x,t}=x+\sqrt{t}z$ for $(x,t)\in\mathbb{R}\times[0,1].$
For a given $\phi\in\mathcal{C}$ and a number $0\leq m\leq 1$, we will simply use $F_{\phi,m}$ to denote $F_{\phi,a}$ if $a\in\mathcal{M}$ is identically equal to $m.$ In this case, $F_{\phi,m}$ can be explicitly written as
\begin{align}\label{eq1}
F_{\phi,m}(x,t)=\frac{1}{m}\log \e\exp m\phi(z_{x,t})=\frac{1}{m}\log \e\exp m\phi(x+\sqrt{t}z).
\end{align}
The central rhythm of the proof for Theorem \ref{thm2} is played by the following proposition.

\begin{proposition}\label{thm1}
Suppose that $(\phi_1,\phi_2)\in \mathcal{F}.$ Let $0<m_1\leq m_2$ and $\alpha\in[0,1]$. Set 
\begin{align*}
n&=\alpha m_1+(1-\alpha)m_2,\\
\phi&=\alpha \phi_1+(1-\alpha)\phi_2.
\end{align*}
Then we have
\begin{align}\label{thm1:eq2}
\forall t\in[0,1],\,\,(F_{\phi_1,m_1}(\cdot,t),F_{\phi_2,m_2}(\cdot,t))\in \mathcal{F}_i\,\,\mbox{if $(\phi_1,\phi_2)\in \mathcal{F}_i$}
\end{align}
for $i=1,2$ and  
\begin{align}
\begin{split}\label{thm1:eq1}
F_{\phi,n}(x,t)&\leq \alpha F_{\phi_1,m_1}(x,t)+(1-\alpha)F_{\phi_2,m_2}(x,t)
\end{split}
\end{align}
for all $(x,t)\in\mathbb{R}\times[0,1].$ 
\end{proposition}

\begin{proof}[Proof of Theorem \ref{thm2}] By the virtue of the uniform $L^1$-Lipschitz property of $P_{\phi,x}$ on $\mathcal{M}$ over $x$, it suffices to assume that $a_1,a_2$ are step functions. Furthermore, we may assume without loss of generality that they jump simultaneously at $\{t_j\}_{j=0}^k$ for some $k\geq 1$, where $0<t_j<t_{j+1}<1$ for all $1\leq j\leq k.$  Let $t_0=0$ and $t_{k+1}=1.$ For $0\leq j\leq k$, let $m_{l,j}=a_l(t_j)$ for $l=1,2$ and $n_j=a(t_j)=\alpha m_{1,j}+(1-\alpha)m_{2,j}.$ Using \eqref{eq1} and an induction argument on $j$, $F_{\phi_l,a_l}$ and $F_{\phi,a}$ can be solved explicitly as
\begin{align}
\begin{split}
\label{eq2}
F_{\phi_l,a_l}(x,t)&=\frac{1}{m_{l,j}}\log \e\exp m_{l,j} F_{\phi_l,a_l}(x+\sqrt{t-t_j}z,t_j),\\
F_{\phi,a}(x,t)&=\frac{1}{n_{j}}\log \e\exp n_{j} F_{\phi,a}(x+\sqrt{t-t_j}z,t_j),
\end{split}
\end{align} 
whenever $(x,t)\in\mathbb{R}\times(t_j,t_{j+1}]$ for $0\leq j\leq k.$ Suppose for the moment that there exists some $0\leq j\leq k$ such that
\begin{align}
\begin{split}\label{eq-1}
&(F_{\phi_1,a_1}(\cdot,t_j),F_{\phi_2,a_2}(\cdot,t_j))\in\mathcal{F}_i,
\end{split}\\
\begin{split}\label{eq-2}
&F_{\phi,a}(\cdot,t_j)\leq \alpha F_{\phi_1,a_1}(\cdot,t_j)+(1-\alpha)F_{\phi_2,a_2}(\cdot,t_j).
\end{split}
\end{align}
Note that $m_{1,j}\leq m_{2,j}$ for all $0\leq j\leq k+1.$
Keeping the iteration equations \eqref{eq2} in mind and using \eqref{eq-1}, we apply Proposition \ref{thm1} with $F_{\phi_1,a_1}(\cdot,t_j),$ $F_{\phi_2,a_2}(\cdot,t_j),$
$ m_{1,j},$ $m_{2,j}$ and  $\Delta t_j:=t_{j+1}-t_j$ to get by \eqref{thm1:eq2},
$$
(F_{\phi_1,a_1}(\cdot,t_{j+1}),F_{\phi_2,a_2}(\cdot,t_{j+1}))\in\mathcal{F}_i$$ 
and by \eqref{eq-2} and then \eqref{thm1:eq1},
\begin{align*}
F_{\phi,a}(x,t_{j+1})
&=\frac{1}{n_j}\log \e\exp n_jF_{\phi,a}(x+\sqrt{\Delta t_j}z,t_j)\\
&\leq\frac{1}{n_j}\log \e \exp n_j\left(\alpha F_{\phi_1,a_1}(x+\sqrt{\Delta t_j}z,t_j)\right.\\
&\qquad\left.+(1-\alpha)F_{\phi_2,a_2}(x+\sqrt{\Delta t_j}z,t_j)\right)\\
&\leq \frac{\alpha}{m_{1,j}}\log \e\exp m_jF_{\phi_1,a_1}(x+\sqrt{\Delta t_j}z,t_j)\\
&\qquad+\frac{1-\alpha}{m_{2,j}}\log \e\exp m_2F_{\phi_1,a_2}(x+\sqrt{\Delta t_j}z,t_j)\\
&=\alpha F_{\phi_1,a_1}(x,t_{j+1})+(1-\alpha)F_{\phi_2,a_2}(x,t_{j+1}).
\end{align*}
From this and Proposition \ref{thm1}, an induction argument leads to $$
F_{\phi,a}(\cdot,t_{k+1})\leq \alpha F_{\phi_1,a_1}(\cdot,t_{k+1})+(1-\alpha)F_{\phi_2,a_2}(\cdot,t_{k+1}),
$$
which gives \eqref{thm2:eq1} and completes our proof. 
\end{proof}

The rest of the paper is devoted to proving Proposition \ref{thm1} that will be divided into two parts. First, we prove \eqref{thm1:eq2}. We begin with a lemma below that gathers a few properties about the expectations for functions of Gaussian random variables as well as two covariance inequalities, the first is a special case of the FKG inequality and the second is taken from \cite{P08}.

\begin{lemma}\label{lem2} 
Suppose that $f,f_1,f_2$ are real-valued functions on $\mathbb{R}$ and $g$ is a centered Gaussian random varaible with $\e g^2=\sigma^2.$ 
\begin{itemize}
\item[$(i)$] If $f$ is even, then $\e f(x+g)$ is even in $x.$
\item[$(ii)$] If $f_1,f_2$ are odd with $f_1\leq f_2$ on $[0,\infty)$ then $\e f_1(x+g)\leq \e f_2(x+g)$ for all $x\geq 0.$
\item[$(iii)$] Let $W$ be a nonnegative function on $\mathbb{R}^2$ with $\e W(x,x+g)=1$ for any $x.$ If $f_1,f_2$ are nondecreasing, then for any $x,$
\begin{align}
\begin{split}\label{lem2:eq1}
&\e f_1(x+g)f_2(x+g)W(x,x+g)\\
&\geq \e f_1(x+g)W(x,x+g)\e f_2(x+g)W(x,x+g).
\end{split}
\end{align}
If
\begin{align}
\label{lem2:cond1}
\left\{
\begin{array}{l}
\mbox{$f_1$ is even with $f_1'(x)\geq 0$ for $x\geq 0$},\\
\mbox{$f_2$ is odd with $f_2'(x)\geq 0$ for $x\geq 0$},\\
\mbox{$W(x,y)$ is even in $y$,}
\end{array}
\right.
\end{align}
then \eqref{lem2:eq1} holds for any $x\geq 0.$
\end{itemize}
\end{lemma}

\begin{proof} Since $g$ and $-g$ have the same distribution, $(i)$ follows by $\e f(x+g)=\e f(-x-g)=\e f(-x+g).$ As for $(ii),$ note that
\begin{align*}
\e f_l(x+g)&=\int_{-\infty}^\infty f_l(u)\rho(u)du=\int_{-\infty}^\infty f_l(u)\rho(u,x)\exp\bigl(\frac{ux}{\sigma^2}\bigr)du,
\end{align*}
where $\rho$ is the probability density of $g$ and $\rho(u,x)=(2\pi\sigma^2)^{-1/2}\exp(-(u^2+x^2)/2\sigma^2).$ If we first split this integral into two parts $[0,\infty)$ and $(-\infty,0]$ and then using change of variables $v=-u$ and the assumption that $f_l$ is odd for the integral on $(-\infty,0]$, it follows that  
\begin{align*}
\e f_l(x+g)&=\int_{0}^\infty f_l(u)\rho(u,x)\exp\bigl(\frac{ux}{\sigma^2}\bigr)du-\int_{0}^{\infty} f_l(v)\rho(v,x)\exp\bigl(-\frac{vx}{\sigma^2}\bigr)dv\\
&=2\int_0^\infty f_l(u)\rho(u,x)\sinh\bigl(\frac{ux}{\sigma^2}\bigr)du.
\end{align*}
Since $\sinh(ux)\geq 0$ for $x,u\geq 0$ and $f_1\leq f_2,$ this equation gives $(ii).$

\smallskip

Next we prove $(iii).$ Let $g'$ be an independent copy of $g.$ Denote $g_x=x+g$ and $g_x'=x+g'.$ Using $\e W(x,g_x)=\e W(x,g_x')=1,$ we write
\begin{align*}
&\e f_1(g_x)f_2(g_x)W(x,g_x)
-\e f_1(g_x)W(x,g_x)\e f_2(g_x)W(x,g_x)\\
&=\e W(x,g_x)W(x,g_x')(f_1(g_x)-f_1(g_x'))(f_2(g_x)-f_2(g_x'))I(g\geq g').
\end{align*}
Applying change of variables $(s,t)=(g_x,g_x'),$ this integral equals
\begin{align}\label{lem2:proof:eq2}
\int_{\{s\geq t\}} K(x,s,t)\exp\biggl(-\frac{1}{2\sigma^2}((s-x)^2+(t-x)^2)\biggr)dsdt,
\end{align}
where 
$$
K(x,s,t):=\frac{1}{2\pi \sigma^2} W(x,s)W(x,t)(f_1(s)-f_1(t))(f_2(s)-f_2(t)).
$$
If $f_1,f_2$ are nondecreasing, this implies that $K$ is nonnegative for any $x$ and the first assertion follows immediately. Assume that $f_1,f_2$ satisfy \eqref{lem2:cond1}. Let us split the integral region of \eqref{lem2:proof:eq2} into two parts $\Omega_1=\{(s,t):s\geq t,|s|\leq |t|\}$ and $\Omega_2=\{(s,t):s\geq t,|s|\geq |t|\}.$ Using change of variables $(u,v)=(-t,-s)$ and the assumptions that $f_1$ is even, $f_2$ is odd and $W(x,y)$ is even in $y$, we obtain
\begin{align*}
&\int_{\Omega_2} K(x,s,t)\exp\biggl(-\frac{1}{2\sigma^2}((s-x)^2+(t-x)^2)\biggr)dsdt\\
&=-\int_{\Omega_1} K(x,u,v)\exp\biggl(-\frac{1}{2\sigma^2}((u+x)^2+(v+x)^2)\biggr)dudv
\end{align*}
and thus, \eqref{lem2:proof:eq2} becomes $\int_{\Omega_1}K(x,s,t)L(x,s,t)dsdt,$
where 
\begin{align*}
L(x,s,t)&:=\exp\biggl(-\frac{1}{2\sigma^2}((s-x)^2+(t-x)^2)\biggr)\\
&\quad-\exp\biggl(-\frac{1}{2\sigma^2}((s+x)^2+(t+x)^2)\biggr).
\end{align*}
Note that since $f_1',f_2'\geq 0$ for $x\geq 0$, it implies $K\geq 0$ on $\Omega_1.$ Also since $x\geq 0$ and $s+t\geq 0$ on $\Omega_1,$ this gives
$
(s+x)^2+(t+x)^2-(s-x)^2-(t-x)^2= 4x(s+t)\geq 0.
$
Together these ensure that $L\geq 0$ on $\Omega_1$ and so \eqref{lem2:eq1} holds for $x\geq 0.$
\end{proof}

\begin{proof}[Proof of \eqref{thm1:eq2} in Proposition \ref{thm1}] For notational convenience, we will denote $F_{\phi_l,m_l}$ by $F_l$ for $l=1,2$ and $F_{\phi,n}$ by $F_0.$ Note that these functions are clearly twice differentiable. Computing directly from \eqref{eq1} yields that
\begin{align}
\begin{split}\label{proof:eq1}
\partial_{x}F_l(x,t)&=\e \phi_l'(z_{x,t})W_l(x,t),
\end{split}\\
\begin{split}\label{proof:eq2}
\partial_{xx}F_l(x,t)&=\e \phi_l''(z_{x,t})W_l(x,t)\\
&\quad+m_l\bigl(\e \phi_l'(z_{x,t})^2W_l(x,t)-\left(\e \phi_l'(z_{x,t})W_l(x,t)\right)^2\bigr),
\end{split}
\end{align}
where $W_l(x,t):=\exp m_l\phi_l(z_{x,t})/\e \exp m_l\phi_l(z_{x,t}).$ Since $\phi_l$ satisfies \eqref{eq-9}, the above two equations imply $F_l(\cdot,t)$ satisfies \eqref{eq-9} too. So $F_l(\cdot,t)\in \mathcal{C}.$ From the right-hand side of \eqref{proof:eq2}, the first term is nonnegative since $\phi_l''\geq 0$, while the second term is nonnegative as well by noting $\e W_l(x,t)=1$ and using Jensen's inequality. Thus, $\partial_{xx}F_l\geq 0$ and $F_l$ is convex in $x$. Note that if $\phi_1\leq \phi_2,$ then the assumption $0\leq m_1\leq m_2$ combining with Jensen's inequality gives
\begin{align*}
F_1(x,t)&=\frac{1}{m_1}\log \e\exp m_1\phi_1(z_{x,t})\\
&\leq \frac{1}{m_1}\log \e\exp m_1\phi_2(z_{x,t})\\
&\leq  \frac{1}{m_2}\log \e\exp m_2\phi_2(z_{x,t})=F_2(x,t).
\end{align*}
Now, on the one hand, if $(\phi_1,\phi_2)\in\mathcal{F}_1,$ then $\phi_1',\phi_2'\geq 0$ and \eqref{proof:eq1} yields that $F_1(\cdot,t),F_2(\cdot,t)$ are nondecreasing. On the other hand, if $(\phi_1,\phi_2)\in\mathcal{F}_2,$ then $\exp m_1\phi_1,\exp m_2\phi_2$ are even and Lemma \ref{lem2}(i) shows that $F_1(\cdot,t),F_2(\cdot,t)$ are also even.

\smallskip

To finish the proof of \eqref{thm1:eq2}, it remains to show that $\partial_xF_1\leq \partial_xF_2$ if $(\phi_1,\phi_2)\in \mathcal{F}_1$ and $\partial_xF_1\leq \partial_xF_2$ for $x\geq 0$ if $(\phi_1,\phi_2)\in \mathcal{F}_2.$ Set $$
G(s)=\e \phi_1'(z_{x,t})W(s,x,z_{x,t}),
$$
where 
$$
W(s,x,y):=\frac{\exp((1-s)m_1\phi_1(y)+sm_2\phi_2(y))}{\e \exp((1-s)m_1\phi_1(z_{x,t})+sm_2\phi_2(z_{x,t}))}.
$$
Then
\begin{align*}
G'(s)&=\e f_1(z_{x,t})f_2(z_{x,t})W(s,x,z_{x,t})\\
&\qquad-\e f_1(z_{x,t})W(s,x,z_{x,t})\e f_2(z_{x,t})W(s,x,z_{x,t}),
\end{align*}
where $f_1:=\phi_1'$ and $f_2:=m_2\phi_2-m_1\phi_1.$ Suppose that $(\phi_1,\phi_2)\in\mathcal{F}_1$. Since $\phi_1$ is convex, $m_2\geq m_1> 0$ and $\phi_2'\geq \phi_1',$ it follows that $f_1$ and $f_2$ are nondecreasing and the first assertion of Lemma \ref{lem2}(iii) yields $G'(s)\geq 0.$ Consequently, using $\phi_1'\leq \phi_2'$ and \eqref{proof:eq1} gives $\partial_xF_1(x,t)=G(0)\leq G(1)\leq \partial_xF_2(x,t)$. Now, if $(\phi_1,\phi_2)\in \mathcal{F}_2$, then $f_1$ is odd, $f_2$ is even and both of them have nonnegative derivatives on $[0,\infty).$ The second assertion of Lemma \ref{lem2}(iii) implies $G'(s)\geq 0$ and thus, 
$
\partial_xF_1(x,t)=G(0)\leq G(1).
$
Note that $\phi_l'(\cdot)\exp m_2\phi_2(\cdot)/\e\exp m_2\phi_2(z_{x,t})$ is odd for $l=1,2$. Using $\phi_1'\leq \phi_2'$ on $[0,\infty)$ and Lemma \ref{lem2}(ii) give
$
G(1)=\e \phi_1'(z_{x,t})W_2(x,t)\leq \e\phi_2'(z_{x,t})W_2(x,t)=\partial_xF_2(x,t)
$
for all $x\geq 0.$ Thus, we conclude $\partial_xF_1(x,t)\leq \partial_xF_2(x,t)$ for all $x\geq 0.$ This finishes the proof of \eqref{thm1:eq2}.
\end{proof}

Next we turn to the proof of \eqref{thm1:eq1} in Proposition \ref{thm1} that relies on the following

\begin{lemma}[Maximum principle]\label{max}
Let $F$ be a twice differentiable function defined on $\mathbb{R}\times[0,1]$ and satisfy the statement:
\begin{align}
\label{lem5:S}
\begin{split}
&\mbox{whenever there is some $(x,t)$ satisfying $\partial_{xx}F(x,t)\leq 0$,}\\
&\mbox{ $\partial_xF(x,t)=0$ and $F(x,t)\geq 0$, then $\partial_tF(x,t)\leq 0.$}
\end{split}
\end{align}
If
\begin{align}
\begin{split}\label{lem5:eq1}
\limsup_{|x|\rightarrow\infty}\sup_{0\leq t\leq 1}F(x,t)&\leq 0,
\end{split}\\
\begin{split}\label{lem5:eq2}
F(\cdot,0)&\leq 0,
\end{split}
\end{align} 
then $F(\cdot,t)\leq 0$ for all $0\leq t\leq 1.$
\end{lemma}

\begin{proof}
For an arbitrary $\varepsilon>0$, set $F_\varepsilon(x,t)=F(x,t)-\varepsilon t$ for $(x,t)\in \mathbb{R}\times[0,1].$ We claim that $F_\varepsilon\leq 0$ on $\mathbb{R}\times[0,1].$ Assume on the contrary that $F_\varepsilon>0$ at $(x_0,t_0).$ From \eqref{lem5:eq1}, there is some $M>0$ such that
$F_\varepsilon(x,t)<F_\varepsilon(x_0,t_0)$ for all $(x,t)\in [-M,M]^c\times[0,1].$ So there exists some $(x_1,t_1)$ that realizes the maximum of $F_\varepsilon$ over $\mathbb{R}\times[0,1]$. Note that from \eqref{lem5:eq2}, $t_1>0.$ At $(x_1,t_1),$ one sees 
\begin{align*}
F-\varepsilon t_1&> 0,\\
\partial_{xx}F&=\partial_{xx}F_\varepsilon\leq 0,\\
\partial_xF&=\partial_xF_\varepsilon=0,\\
\partial_tF-\varepsilon&=\partial_tF_\varepsilon\geq 0.
\end{align*}
The first three lines and the statement \eqref{lem5:S} give $\partial_t F(x_1,t_1)\leq 0.$ However, the last line reads $\partial_tF(x_1,t_1)\geq \varepsilon$, a contradiction. This completes the proof of our claim and consequently, $F(x,t)\leq \varepsilon t$ for all $(x,t)\in\mathbb{R}\times[0,1]$ and $\varepsilon>0.$ Letting $\varepsilon$ tend to zero finishes our proof.   
\end{proof}

Recall $F_0,F_1,F_2$ from the first part of the proof of Proposition \ref{thm1}. As one shall see, we will define $F=F_0-\alpha F_1-(1-\alpha)F_2$ and   use Lemma \ref{max} to show $F\leq 0.$ For technical purposes, we will need two lemmas to simplify our argument. Since their proofs are seemingly independent of our main goal, we will postpone them to the appendix.

\begin{lemma}\label{lem-1}
Suppose that $(\phi_1,\phi_2)\in\mathcal{F}_i.$ Then there exists $\{(\phi_{1,r},\phi_{2,r})\}_{r\geq 1}\subseteq \mathcal{F}_i$ such that for $l=1,2,$ 
$\phi_{l,r}$ is linear on $(-\infty,-M_{r}]\cup[M_{r},\infty)$ for some $M_{r}> 0$, $\phi_{l,r}\leq \phi_{l}$ and $\phi_{l,r}\rightarrow\phi_l$ pointwise.
\end{lemma}

\begin{lemma}\label{lem--2}
Let $m,M\geq 0.$ Suppose that $\phi$ is a continuous function on $\mathbb{R}$ with $\phi(x)=Ax+B$ for $x\geq M$ and $\phi(x)=A'x+B'$ for $x\leq -M,$ where $A,A',B,B'$ are constants. Then we have
\begin{align}\label{eq--5}
(\e\exp m\phi(x+\sqrt{t}z))^{\frac{1}{m}}=\left\{
\begin{array}{ll}
O(x,t)\exp \bigl(Ax+B+\frac{A^2mt}{2}\bigr),\\
\\
O'(x,t)\exp \bigl(A'x+B'+\frac{(A')^2mt}{2}\bigr),
\end{array}\right.
\end{align}
where $\lim_{x\rightarrow\infty}O(x,t)=1$ and $\lim_{x\rightarrow-\infty}O'(x,t)=1$ uniformly over $0\leq t\leq 1$.
\end{lemma}

%\begin{lemma}\label{lem1}
%Suppose that $c_1\geq c_2$ and $c_2\leq 0.$ Then 
%\begin{align}
%\label{lem1:eq1}
%f(y_1,y_2):=(y_1-y_2)(c_1y_1-c_2y_2)\leq 0
%\end{align}
%for all $(y_1,y_2)\in\mathbb{R}^2$ satisfying either $0\leq y_1\leq y_2$ or $y_2\leq y_1\leq 0.$
%\end{lemma}

%\begin{proof}
%Since $f(y_1,y_2)=f(-y_1,-y_2),$ it suffices to show \eqref{lem1:eq1} by assuming $0\leq y_1\leq y_2.$
%In this case, since $c_1-c_2\geq 0$, $c_2\leq 0,$ $y_1\geq 0$ and $y_1- y_2\leq 0$, we have
%$c_1y_1-c_2y_2=(c_1-c_2)y_1+c_2(y_1-y_2)\geq 0$
%and \eqref{lem1:eq1} follows.
%\end{proof}

\begin{proof}[Proof of \eqref{thm1:eq1} in Proposition \ref{thm1}]
From Lemma \ref{lem-1}, there exists $(\phi_{1,r},\phi_{2,r})\in\mathcal{F}_i$ such that $\phi_{l,r}$ is linear on $(-\infty,-M_r]\cup [M_r,\infty)$, $\phi_{l,r}\leq \phi_l$ and $\phi_{l,r}\rightarrow \phi_l$ pointwise. If we could show that \eqref{thm1:eq1} holds for all $(\phi_{1,r},\phi_{2,r})$, the dominated convergence theorem implies that \eqref{thm1:eq1} is also valid for $(\phi_1,\phi_2).$ Thus, we may assume without loss of generality that $\phi_1,\phi_2$ are linear on $(-\infty,-M]\cup[M,\infty)$ for some $M>0.$ 

\smallskip
Define $
F=F_0-\alpha F_1-(1-\alpha)F_2.
$
Our goal is to show that $F\leq 0$ via Lemma \ref{max}. In order to do so, we now check that the conditions \eqref{lem5:eq1} and \eqref{lem5:eq2} are satisfied and the statement \eqref{lem5:S} holds true as follows. First, \eqref{lem5:eq2} follows immediately form the definitions of $F_0,F_1,F_2.$ Next, we proceed to check \eqref{lem5:eq1}. We show that $\limsup_{x\rightarrow\infty}\sup_{0\leq t\leq 1}F(x,t)\leq 0$ first. From the linearity of $\phi_l$ on $[M,\infty)$, write $\phi_{l}(x)=A_lx+B_l$ for all $x\geq M$ and some $A_l,B_l\in\mathbb{R}.$ Note that no matter $(\phi_1,\phi_2)\in\mathcal{F}_1$ or $\mathcal{F}_2$, we always have $\phi_2'\geq \phi_1'\geq 0$ on $[0,\infty)$. Thus, $A_2\geq A_1\geq 0.$ From Lemma \ref{lem--2},
\begin{align*}
\exp F_0(x,t)&=O(x,t)\exp \biggl(\alpha(A_1x+B_1)+(1-\alpha)(A_2x+B_1)\\
&\qquad\qquad+\frac{(\alpha A_1+(1-\alpha)A_2)^2nt}{2}\biggr),\\
\exp F_l(x,t)&=O_l(x,t)\exp \biggl(A_lx+B_l+\frac{A_l^2m_lt}{2}\biggr),
\end{align*}
where $O,O_1,O_2$ converge to $1$ uniformly over $t\in[0,1]$ as $x$ tends to infinity. So
\begin{align}
\begin{split}\label{eq.2}
\exp F(x,t)
&=\frac{O(x,t)}{O_1(x,t)^{\alpha}O_2(x,t)^{1-\alpha}}\\
&\exp\frac{1}{2}(n(\alpha A_1+(1-\alpha)A_2)^2-\alpha m_1 A_1^2-(1-\alpha)m_2A_2^2).
\end{split}
\end{align}
Here the exponent on the right-hand side can be factorized as
\begin{align}\label{eq.1}
\frac{1}{2}(A_1-A_2)(c_1A_1-c_2A_2)=\frac{1}{2}(A_1-A_2)((c_1-c_2)A_1+c_2(A_1-A_2)),
\end{align}
where $c_1:=\alpha(n\alpha-m_1)$ and $c_2:=(1-\alpha)(n(1-\alpha)-m_2).$ Observe that $c_1-c_2=2\alpha(1-\alpha)(m_2-m_1)$ and that $c_2\geq 0$ if and only if $m_1/m_2\geq (2-\alpha)/(1-\alpha).$ Since $m_2\geq m_1$ and $0\leq \alpha\leq 1,$ one sees that $c_1\geq c_2$ and $c_2<0.$ From the right-hand side of \eqref{eq.1}, these combining with $0\leq A_1\leq A_2$ give that $\eqref{eq.1}\leq 0$ and so $\limsup_{x\rightarrow \infty}\sup_{0\leq t\leq 1}F(x,t)\leq 0$. Next, we check $\limsup_{x\rightarrow -\infty}\sup_{0\leq t\leq 1}F(x,t)\leq 0.$ If $(\phi_1,\phi_2)\in \mathcal{F}_2,$ this follows immediately by the symmetry of $F$ in $x$. If $(\phi_1,\phi_2)\in \mathcal{F}_1,$ one may write $\phi_{l}(x)=A_l'x+B_l'$ for $x\leq -M$ and argue exactly in the same way as above to obtain \eqref{eq.2} with new parameters $A_l'$ and $B_l'$. In such case, note that again we have $0\leq A_1'\leq A_2'$ since $0\leq \phi_1'\leq \phi_2'$ on $(-\infty,-M]$. As a result, $\eqref{eq.1}\leq 0$ and $\limsup_{x\rightarrow -\infty}\sup_{0\leq t\leq 1}F(x,t)\leq 0.$ Thus, \eqref{lem5:eq1} holds. 

\smallskip

Finally, we claim that the statement \eqref{lem5:S} holds. Assume $\partial_{xx}F\leq 0$, $\partial_xF=0$ and $F\geq 0$ at some $(x_0,t_0).$ Using \eqref{max} and $\partial_xF(x_0,t_0)=0$, we have that at $(x_0,t_0),$
$
\partial_tF=\Delta_1/2+\Delta_2/2,
$
where $\Delta_1:=\partial_{xx}F$ and
\begin{align*}
\Delta_2&:=n\left(\alpha\partial_xF_1+(1-\alpha)\partial_xF_2\right)^2-\alpha m_1(\partial_xF_1)^2-(1-\alpha)m_2(\partial_xF_2)^2.
\end{align*}
Note that from assumption, $\Delta_1\leq 0.$ As for $\Delta_2,$ it can be factorized as
\begin{align}
\begin{split}
\label{eq.3}
\Delta_2&=\left(\partial_xF_1-\partial_xF_2\right)\left(c_1\partial_xF_1-c_2\partial_xF_2\right)\\
&=\left(\partial_xF_1-\partial_xF_2\right)\left((c_1-c_2)\partial_xF_1+c_2(\partial_xF_1-\partial_xF_2)\right),
\end{split}
\end{align}
where $c_1,c_2$ are defined in \eqref{eq.1}. Note that $c_1,c_2$ satisfy $c_2\leq c_1$ and $c_2<0.$ Also from \eqref{thm1:eq2}, we have that $0\leq \partial_{x}F_1\leq \partial_xF_2$ if $(\phi_1,\phi_2)\in\mathcal{F}_1$ and that $0\leq \partial_xF_1\leq \partial_xF_2$ for $x\geq 0$ and $\partial_xF_2\leq \partial_xF_1\leq 0$ for $x\leq 0$ if $(\phi_1,\phi_2)\in\mathcal{F}_2$. Thus, from the right-hand side of \eqref{eq.3}, combining these together yields $\Delta_2\leq 0$ and then $\partial_tF(x_0,t_0)\leq 0,$ which means that the statement \eqref{lem5:S} is satisfied. This completes our proof.
\end{proof}

% % % % % % % % % % % % % % % % % % % % % % % % % % % % % % % % % % % % % % % % % % % % % % % % % % % % % % % % % % % % % % % % % % % % % % % % % % % % %
% % % % % % % % % % % % % % % % % % % % % % % % % % % % % % % % % % % % % % % % % % % % % % % % % % % % % % % % % % % % % % % % % % % % % % % % % % % % 

\noindent{\bf \Large Appendix}

\begin{proof}[Proof of Lemma \ref{lem-1}]
For every $r\geq 1,$ let $T_r$ be the smallest integer such that 
\begin{align}
\label{eq--2}
\max\{\max\{|\phi_1'(x)|:x\in[-r,r]\},\max\{|\phi_2'(x)|:x\in[-r,r]\}\}\leq T_r.
\end{align}
Set $q_{r,p}=p(rT_r)^{-1}$ for $-r^2T_r\leq p\leq r^2T_r.$ In other words, $q_{r,p}$'s form a regular partition of $[-r,r]$ with $2r^2T_r+1$ points. Let $s_{l,r}$ be a continuous piecewise linear function on $\mathbb{R}$ defined as
\begin{align}\label{eq--3}
s_{l,r}(x)=\left\{
\begin{array}{ll}
\phi_l(r)-\frac{2}{r}+\phi_{l}'(r)(x-r),&\mbox{if $x\geq r$},\\
\\
\phi_l(q_{r,p})-\frac{2}{r}+\frac{\phi_{l}(q_{r,p+1})-\phi_l(q_{r,p})}{q_{r,p+1}-q_{p}}(x-q_{r,p}),&\mbox{if $q_{r,p}<x\leq q_{r,p+1}$},\\
\\
\phi_l(-r)-\frac{2}{r}+\phi_l'(-r)(x+r),&\mbox{if $x\leq -r$}.
\end{array}\right.
\end{align}
It is easy to check that $s_{l,r}$ is convex and $s_{1,r}\leq s_{2,r}$. In addition, if $(\phi_1,\phi_2)\in\mathcal{F}_1,$ then $s_{1,r},s_{2,r}$ are nondecreasing and $s_{1,r}'\leq s_{2,r}'$ on $\mathbb{R}$ except at $q_{r,p}$'s; if $(\phi_1,\phi_2)\in\mathcal{F}_2,$ then $s_{1,r},s_{2,r}$ are even and $s_{1,r}'\leq s_{2,r}'$ on $[0,\infty)$ except at $q_{r,p}$'s. Note that since the first and third equations on the right-hand side of \eqref{eq--3} are the supporting lines of $\phi_l-2/r$ at $r$ and $-r$, respectively, it follows that $\phi_{l}-2/r\geq s_{l,r}$ for $|x|\geq r.$ On the other hand, if $x\in(q_{r,p},q_{r,p+1}]$ for some $-r^2T_r\leq p\leq r^2T_r-1,$ then by the convexity of $\phi_l,$
\begin{align}
\begin{split}
\label{eq.5}
\phi_l(x)-s_{l,r}(x)&=\phi_l(x)-\biggl(\phi_l(q_{r,p})-\frac{2}{r}+\frac{\phi_{l}(q_{r,p+1})-\phi_l(q_{r,p})}{q_{r,p+1}-q_{r,p}}(x-q_{r,p})\biggr)\\
&=(x-q_{r,p})\biggl(\frac{\phi_l(x)-\phi_l(q_{r,p})}{x-q_{r,p}}-\frac{\phi_{l}(q_{r,p+1})-\phi_l(q_{r,p})}{q_{r,p+1}-q_{r,p}}\biggr)+\frac{2}{r}\\
&\geq (x-q_{r,p})(\phi_l'(q_{r,p})-\phi_l'(q_{r,p+1}))+\frac{2}{r}\\
&\geq -(q_{r,p+1}-q_{r,p})T_r+\frac{2}{r}\\
&\geq \frac{1}{r}.
\end{split}
\end{align}
To sum up, $\phi_l\geq s_{l,r}.$ Now, we pick a symmetric mollifier function $\eta$ on $\mathbb{R}$ that is positive on $(-1,1)$ and supported on $[-1,1]$. Define $\eta_r(x)=\varepsilon_r^{-1}\eta(x/\varepsilon_r)$ for $\varepsilon_r:=(2rT_r)^{-1}.$ We set 
\begin{align}
\label{eq-5}
\phi_{l,r}(x):=\int \eta_{r}(x-u)s_{l,r}(u)du=\int\eta_{r}(u)s_{l,r}(x-u)du.
\end{align}
With this definition, one may see that $\phi_{l,r}$'s are twice differentiable functions and clearly they satisfy \eqref{eq-9}.

\smallskip

Now, we check that $(\phi_{1,r},\phi_{2,r})\in\mathcal{F}_i.$ Let us consider $i=1$ first. Since $s_{1,r}$ and $s_{2,r}$ are convex, nondecreasing and $s_{1,r}\leq s_{2,r}$, using the equation on the right-hand side of \eqref{eq-5}, it is easy to see that these properties are also preserved for $\phi_{1,r},\phi_{2,r}$. Note that from our construction of $s_{l,r}$, $|s_{l,r}(x)-s_{l,r}(y)|\leq T_r|x-y|$. This together with \eqref{eq.5} and $\mbox{supp}(\eta_{r})=[-\varepsilon_r,\varepsilon_r]$ yields
\begin{align}
\begin{split}\label{eq--1}
\phi_{l}(x)-\phi_{l,r}(x)&=\phi_l(x)-s_{l,r}(x)+s_{l,r}(x)-\phi_{l,r}(x)\\
&=\phi_l(x)-s_{l,r}(x)+\int \eta_{r}(u)(s_{l,r}(x)-s_{l,r}(x-u))du\\
&\geq \frac{1}{r}-T_r\int \eta_{r}(u)|u|du\\
&\geq \frac{1}{r}-T_r\varepsilon_r\int \eta_{r}(u)du\\
&\geq \frac{1}{2r}.
\end{split}
\end{align}
So $\phi_l\geq \phi_{l,r}.$ Finally, it remains to check that $\phi_{1,r}'\leq \phi_{2,r}'.$ To see this, observe that $s_{l,r}$ is differentiable everywhere except at $q_{r,p}$'s and is Lipschitz on $\mathbb{R}.$ From the dominated convergence theorem,
\begin{align}\label{eq--4}
\phi_{1,r}'(x)=\int \eta_{r}(u)s_{1,r}'(x-u)du\leq \int \eta_{r}(u)s_{2,r}'(x-u)du=\phi_{2,r}'(x)
\end{align}
for every $x\in\mathbb{R}.$ So $\phi_{1,r}'\leq \phi_{{2,r}}'$ on $\mathbb{R}.$ This completes our proof of $(\phi_{1,r},\phi_{2,r})\in\mathcal{F}_1.$ As for the case that $(\phi_{1,r},\phi_{2,r})\in\mathcal{F}_2$ by assuming $(\phi_1,\phi_2)\in\mathcal{F}_2$, to show that $\phi_{1,r},\phi_{2,r}$ are convex and $\phi_{1,r}\leq \phi_{2,r}$, it can be treated essentially in the same way as above. Moreover, using the additional assumption that $\eta$ is symmetric gives that $\phi_{1,r},\phi_{2,r}$ are even. As for $\phi_{1,r}'\leq \phi_{2,r}'$ on $[0,\infty),$ note that $s_{2,r}'\geq s_{1,r}'$ on $[0,\infty)$ and $\eta_r$ is supported on $[-\varepsilon_r,\varepsilon_r]$. If $x\geq \varepsilon_r,$ then we also have \eqref{eq--4}; if $x\in[0,\varepsilon_r),$ noting that $\varepsilon_r\leq q_{r,1}$ and $s_{r,l}(u)=A_l|u|+B_l$ with $A_l=(\phi_{l}(q_{r,1})-\phi_l(q_{r,0}))(q_{r,1}-q_{r,0})^{-1}$ and $B_l=\phi_l(0),$ we compute
\begin{align*}
\phi_{l,r}'(x)&=\int \eta_{r}(u)s_{l,r}'(x-u)du\\
&=A_l\biggl(\int_{-\varepsilon_r}^x \eta_{r}(u)du-\int_{x}^{\varepsilon_r} \eta_{r}(u)du\biggr)\\
&=A_l\biggl(1-2\int_{x}^{\varepsilon_r} \eta_{r}(u)du\biggr).
\end{align*}
Since $\int_{x}^{\varepsilon_r} \eta_{r}(u)du\leq 1/2$ and $A_1\leq A_2,$ they imply $\phi_{1,r}'(x)\leq \phi_{2,r}'(x).$ Hence, $\phi_{1,r}'\leq \phi_{2,r}'$ on $[0,\infty).$ This completes the proof for the second case.

Next, we show that $\phi_{l,r}$ is linear on $(-\infty,-M_{r}]\cup[M_{r},\infty)$ for some $M_{r}>0$, $\phi_{l,r}\leq \phi_l$ and $\phi_{l,r}\rightarrow\phi_l.$ From \eqref{eq--1}, $\phi_{l,r}\leq \phi_l$ holds. Also, from the second equality of \eqref{eq--1}, one can further see that
\begin{align*}
|\phi_l(x)-\phi_{l,r}(x)|&\leq \phi_l(x)-s_{l,r}(x)+J_{l,r}\int \eta_{r}(u)|u|du\\
&\leq \phi_l(x)-s_{l,r}(x)+\frac{1}{2r}.
\end{align*}
Since $s_{l,r}\uparrow \phi_l,$ it follows that $\phi_{l,r}\rightarrow\phi_l.$ Finally, recalling the definition of $s_{l,r}$ from \eqref{eq--3}, if we pick $M_{r}=r+\varepsilon_{r},$ then
\begin{align*}
&\phi_{l,r}(x)\\
&=\int\eta_{r}(u)s_{l,r}(x-u)du\\
&=
\left\{
\begin{array}{ll}
\bigl(\phi_{l}(r)-\frac{2}{r}-\phi_l'(r)\int\eta_{r}(u)udu\bigr)+\bigl(\phi_l'(r)\int \eta_r(u)du\bigr)x,&\mbox{if $x\geq M_r$},\\
\\
\bigl(\phi_{l}(-n)-\frac{2}{r}-\phi_l'(-r)\int\eta_{r}(u)udu\bigr)+\bigl(\phi_l'(-r)\int \eta_r(u)du\bigr)x,&\mbox{if $x\leq -M_r$},
\end{array}
\right.
\end{align*}
which means that $\phi_{l,r}$ is linear on $(-\infty,-M_{r}]\cup[M_{r},\infty).$ This completes our proof.
\end{proof}

\begin{proof}[Proof of Lemma \ref{lem--2}] We will only show the first equality of the right-hand side of \eqref{eq--5}. For the second, it can be derived exactly in the same way. For notational convenience, we will use $C$ to denote a positive constant independent of $x,t$. Let us emphasize that it might be different from each occurrence. Denote $z_{x,t}=x+\sqrt{t}z.$ From the given assumption, we write
\begin{align}
\begin{split}
\label{eq-7}
\e\exp m\phi(z_{x,t})&=I_1(x,t)+I_2(x,t)+I_3(x,t)
\end{split}
\end{align}
where
\begin{align*}
I_1(x,t)&:=\e \exp m(Az_{x,t}+B),\\
I_2(x,t)&:=\e[\exp m\phi(z_{x,t});|z_{x,t}|\leq M],\\
I_3(x,t)&:=\e[\exp m(A'z_{x,t}+B');z_{x,t}\leq -M]-\e[\exp m(Az_{x,t}+B);z_{x,t}\leq M].
\end{align*}
Since $\phi$ is continuous, $|I_2|\leq C\p(|z_{x,t}|\leq M)$. From the Cauchy-Schwartz inequality,
\begin{align*}
|I_3|&\leq (\e \exp 2m(A'z_{x,t}+B'))^{\frac{1}{2}}\p(z_{x,t}\leq -M)^{\frac{1}{2}}\\
&\quad+(\e \exp 2m(Az_{x,t}+B))^{\frac{1}{2}}\p(z_{x,t}\leq M)^{\frac{1}{2}}\\
&=\exp m(A'x+B'+2m(A')^2t)\p(z_{x,t}\leq -M)^{\frac{1}{2}}\\
&\quad+\exp m(Ax+B+2mA^2t)\p(z_{x,t}\leq M)^{\frac{1}{2}}\\
&\leq C\exp (Cx)\p(z_{x,t}\leq M)^{\frac{1}{2}},
\end{align*}
where we used the fact that $\e \exp \beta z=\exp \beta^2/2$ for any $\beta\in\mathbb{R}.$
Since $\p(z\geq L)\leq \exp(-L^2/2)$ whenever $L$ is sufficiently large, we have
\begin{align}
\label{eq-8}
\p(z_{x,t}\leq M)&=\p(x-M\leq \sqrt{t}z)\leq \exp\biggl(-\frac{(x-M)^2}{2t}\biggr)\leq C\exp\biggl(-\frac{x^2}{Ct}\biggr)
\end{align}
when $x$ is sufficiently large. Therefore, $|I_2|+|I_3|\leq C\exp(-x^2/Ct).$ Finally, noting that $I_1=\exp m(Ax+B+A^2mt/2)$ gives 
$(|I_2|+|I_3|)/I_1\leq C\exp(-x^2/Ct+m(Ax+B)+A^2m^2t/2).$ Since the exponent is dominated by $x^2/t,$ we get again 
$(|I_2|+|I_3|)/I_1\leq C\exp(-x^2/Ct).$ Hence, if we define $O=(1+(I_2+I_3)/I_1)^{1/m},$ then $O$ converges to $1$ uniformly over $0\leq t\leq 1$ as $x$ tends to infinity and $(\e\exp m\phi(z_{x,t}))^{1/m}=O(x,t)I_1(x,t)^{1/m}=O(x,t)\exp m(Ax+B+A^2mt/2).$ This gives the announced result.
\end{proof}

\end{document}